\documentclass[12pt,final]{article}
\usepackage[margin=2.5cm,top=1.5cm]{geometry}
\usepackage[all,arc]{xy}
\usepackage{mathrsfs}
\usepackage{amsmath,amsthm,amssymb,enumerate}
\usepackage{color}

\newtheorem{defn}{Definition}[section]
\newtheorem{lem}[defn]{Lemma}
\newtheorem{theo}[defn]{Theorem}
\newtheorem{cor}[defn]{Corollary}

\newtheorem{prop}[defn]{Proposition}
\newtheorem{error}{Error}[section]
\newtheorem{exam}[error]{Example}

\theoremstyle{remark}

\newcommand{\Romannum}[1]{\uppercase\expandafter{\romannumeral #1}}

\numberwithin{equation}{section}

\newcommand\keywordsname{Key words}
\newcommand\AMSname{AMS subject classifications}

\newenvironment{@abssec}[1]{%
     \if@twocolumn
       \section*{#1}%
     \else
       \vspace{.05in}\footnotesize
       \parindent .2in
         {\upshape\bfseries #1. }\ignorespaces
     \fi}
     {\if@twocolumn\else\par\vspace{.1in}\fi}

\begin{document}

\vskip6cm
\title{Some sharp bounds on the distance signless Laplacian spectral radius of graphs
\footnote{Research supported by National Natural Science Foundation of China
(No. 10901061), the Zhujiang Technology New Star Foundation of
Guangzhou (No. 2011J2200090), and Program on International Cooperation and Innovation, Department of Education, Guangdong Province (No.
2012gjhz0007).}}
\author{Wenxi Hong, Lihua You\footnote{{\it{Corresponding author:\;}}ylhua@scnu.edu.cn.}}
\vskip.2cm
\date{{\small
School of Mathematical Sciences, South China Normal University,\\
Guangzhou, 510631, P.R. China\\
}} \maketitle

\begin{abstract}

 \vskip.3cm
 M. Aouchiche and P. Hansen proposed the distance Laplacian and the distance signless Laplacian of a connected graph [Two Laplacians for the distance matrix of a graph, LAA 439 (2013) 21--33]. In this paper, we obtain three theorems on the sharp upper bounds of the spectral radius of a nonnegative matrix, then apply these theorems to signless Laplacian matrices and the distance signless Laplacian matrices to obtain some sharp bounds on the spectral radius, respectively. We also proposed a known result about the sharp bound of the signless Laplacian spectral radius has a defect.
\end{abstract}

\section{ Introduction}
\hskip.6cm
Let $G$ be a simple graph with vertices set $V(G)=\{v_1, v_2, \ldots, v_n\}$. Let $d_i$ be the degree of the vertex $v_i$ in $G$ for $i=1, 2, \ldots, n$ and satisfy $d_1\geq d_2\geq \cdots\geq d_n$. Let $A(G)=(a_{ij})_{n\times n}$ be the $(0,1)$--adjacency matrix of $G$, where $a_{ij}=1$ if $v_i$ and $v_j$ are adjacent and 0 otherwise, and $D(G)=diag(d_1, d_2, \ldots, d_n)$ be the degree diagonal matrix. Then $Q(G)=D(G)+A(G)$ is the signless Laplacian matrix of $G$.

The distance matrix $\mathcal{D}(G)$ of $G$ is defined so that its $(i, j)$--entry, $d_{ij}$, is equal to $d_G(v_i, v_j)$ which denotes the distance (the length of the shortest path) between $v_i$ and $v_j$. Clearly, the distance matrix of a connected graph is irreducible and symmetric. The eigenvalues of $\mathcal{D}(G)$ are given as $\delta_1(G)\geq\delta_2(G)\geq\cdots\geq\delta_n(G)$, where $\delta_1$ is called the distance spectral radius of $G$.

The transmission of $v_i$ is defined to be the sum of the distance from $v_i$ to all other
vertices in $G$, denoted by $\mathcal{D}_i$, that is, $\mathcal{D}_i=\sum\limits_{j=1}^{n}d_{ij}$. The transmission is also called the first distance degree \cite{2010b}. Assume that the transmissions are ordered as $\mathcal{D}_1\geq \mathcal{D}_2 \geq\cdots \geq \mathcal{D}_n$. A connected graph $G$ is said to be $k$--transmission regular if $\mathcal{D}_i=k$ for all $i\in\{1, 2 ,\ldots, n\}$. Let $Tr(G)=diag(\mathcal{D}_1, \mathcal{D}_2, \ldots, \mathcal{D}_n)$ denote the diagonal matrix of its vertex transmissions.

M. Aouchiche and P. Hansen \cite{2013} introduced the Laplacian and the signless Laplacian for the distance matrix of a connected graph. The matrix $\mathcal{D}^{L}=Tr(G)-\mathcal{D}(G)$ is called the distance Laplacian of $G$, while the matrix $\mathcal{D}^{Q}=Tr(G)+\mathcal{D}(G)$ is called the distance signless Laplacian of $G$.

Let $G$ be a connected graph, then the matrix $\mathcal{D}^{Q}=(q_{ij})$ is symmetric, nonnegative and irreducible. All the eigenvalues of $\mathcal{D}^Q$ can be arranged as:
$\delta_1^Q(G)\geq \delta_2^Q(G)\geq\cdots\geq \delta_n^Q(G)$.
$\delta_1^Q(G)$ is called the distance signless Lplacian radius of $G$. As the $\mathcal{D}^{Q}$ is irreducible, by the Perron--Frobenius theorem,
 $\delta_1^Q(G)$ is positive, simple and there is a unique positive unit eigenvector $x$ corresponding to
$\delta_1^Q(G)$, which is called the distance signless Perron vector of $G$.

As usual, we denote by $K_n$ the complete graph, by $C_n$ the cycle, by $S_n$ the star graph and by $K_{a,n-a}$ the complete bipartite graph. In \cite{2013}, M. Aouchiche and P. Hansen give the lower sharp bound as $\delta_1^Q(G)\geq \delta_1^Q(K_n)=2n-2$.

In this paper, we obtain three theorems on the sharp upper bounds of the spectral radius of a nonnegative matrix, and apply these theorems to the signless Laplacian matrix to get the known results about the signless Laplacian spectral radius of graphs in Section 2. In Section 3, we present some sharp bounds on the distance signless Laplacian spectral radius of graphs.

\section{Three theorems on the spectral radius of a nonnegative matrix}

\hskip.6cm
In this section, we obtain three theorems on the sharp upper bounds of the spectral radius of a nonnegative matrix. Then we apply these theorems to the signless Laplacian matrix and obtain the known results about the signless Laplacian spectral radius of graphs. We also proposed a known result about the sharp bound of the signless Laplacian spectral radius has a defect.
\begin{lem}\label{lem 2.2}{\rm(\cite{1988})}
If $A$ is an $n\times n$ nonnegative matrix with the spectral radius $\lambda(A)$ and row sums $r_1, r_2, \ldots, r_n$, then
$\min\limits_{1\leq i\leq n}r_i\leq \lambda(A)\leq \max\limits_{1\leq i\leq n}r_i$. Moreover, if $A$ is irreducible, then one of the equalities holds if and only if the row sums of $A$ are all equal.
\end{lem}

\begin{theo}\label{thm 3.1}
Let $A$ be an $n\times n$ nonnegative, irreducible and symmetric matrix with row sums $r_1, r_2, \ldots, r_n$, where $r_1\geq r_2\geq\ldots\geq r_n$, $B=A+M$ where $M=diag(r_1, r_2, \ldots, r_n)$, and $\lambda(A)$ $(\lambda(B))$ the largest eigenvalue of $A$ $(B)$. Then
$\lambda(B)\leq \lambda(A)+ r_1$ with equality if and only if the row sums of $A$ are all equal.
\end{theo}
\begin{proof}
Let $x=(x_1, x_2, \ldots, x_n)^T$ be the unit positive vector corresponding to $\lambda(B)$, then
\vskip.2cm
\hskip0.8cm$\lambda(B)=x^{T}Bx$
\vskip.2cm
\hskip1.8cm$=x^{T}(A+M)x$
\vskip.2cm
\hskip1.8cm$=x^{T}Mx+x^{T}Ax$
\vskip.2cm
\hskip1.8cm$=\sum\limits_{i=1}^nx_i^2r_i+x^{T}Ax$
\vskip.2cm
\hskip1.8cm$\leq \sum\limits_{i=1}^nx_i^2r_1+x^{T}Ax$
\vskip.2cm
\hskip1.8cm$\leq r_1+\lambda(A)$.
\vskip.2cm
Moreover, when the equality holds, then the row sums of $A$ are all equal by Lemma \ref{lem 2.2}. Conversely, when the row sums of $A$ are all equal, then $\lambda(B)=2r_1$ and $\lambda(A)=r_1$ by Lemma \ref{lem 2.2}. The equality holds.
\end{proof}

Let $G$ be a simple and connected graph. Let $A$ be the adjacency matrix of $G$ and $M$ be the degree diagonal matrix of $G$. Then $B=A+M$ is the signless Laplacian matrix of $G$. So we can get the following corollary by Theorem \ref{thm 3.1}.
\begin{cor}\label{cor 3.1}{\rm(\cite{2010a}, Lemma 9)}
Let $G$ be a simple and connected graph on $n$ vertices, $\lambda_1(G)$ be the spectral radius of $G$, $q_1(G)$ be the signless Laplacian spectral radius of $G$ and $\Delta$ be the maximum degree of $G$. Then $q_1(G)\leq\lambda_1(G)+\Delta$ with equality if and only if $G$ is regular.
\end{cor}

\begin{defn}\label{defn 2.2}{\rm(\cite{1979})}
 Let $(a)=(a_1, a_2, \ldots, a_r)$ and $(b)=(b_1, b_2, \ldots, b_s)$ are nonincreasing
\vskip.1cm
\noindent sequences of real numbers. Then $(a)$ majorizes $(b)$ if $(a)$ and $(b)$ satisfy the  two conditions:

\hskip.8cm\rm{$(1)$}$\sum\limits_{i=1}^ka_i\geq\sum\limits_{i=1}^kb_i$, where $k=1,2,\ldots,\min\{r, s\}$,

\hskip.8cm\rm{$(2)$}$\sum\limits_{i=1}^ra_i=\sum\limits_{i=1}^sb_i$.
\end{defn}

\begin{lem}\label{lem 2.4}{\rm(\cite{1923})}
Let $A$ be a positive semidefinite Hermitian matrix, $\lambda_1, \lambda_2, \ldots, \lambda_n$ be the eigenvalues of $A$ where $\lambda_1\geq\lambda_2\geq\cdots\geq\lambda_n$ and $a_1, a_2, \ldots, a_n$ be the entries of the main diagonal of $A$ satisfying $a_1\geq a_2\geq \cdots\geq a_n$.
Then the spectrum of $A$ majorizes its main diagonal, that is, $\sum\limits_{i=1}^k\lambda_i\geq \sum\limits_{i=1}^ka_i$ for $k=1, 2, \ldots, n$, and $\sum\limits_{i=1}^n\lambda_i=\sum\limits_{i=1}^na_i$.
\end{lem}

\begin{theo}\label{thm 3.4}
Let $A=(a_{ij})$ be an $n\times n$ nonnegative and irreducible matrix with $a_{ii}=0$ for $i=1, 2, \ldots, n$
and row sums $r_1, r_2, \ldots, r_n$ with $r_1\geq r_2\geq\ldots\geq r_n$. Let $B=A+M$, where $M=diag(r_1, r_2, \ldots, r_n)$, $s_i=\sum\limits_{j=1}^na_{ij}r_j$  for any $i\in\{1, 2, \ldots, n\}$ and $\lambda(B)$ be the spectral radius of $B$. If $B$ is a positive semidefinite Hermitian matrix, then
\vskip.2cm
\hskip0.8cm $\lambda(B)\leq \max\limits_{1\leq i,j\leq n}\left\{\frac{r_i+r_j+\sqrt{(r_i-r_j)^2+\frac{4s_is_j}{r_ir_j}}}{2}\right\}$,

\noindent with equality if and only if row sums of $A$ are all equal.
\end{theo}
\begin{proof}
Note that $B=(b_{ij})$, where
$b_{ij}=\left\{
\begin{array}{cc}
 r_{i}, & \mbox{if } i=j; \\[0.2mm]
 a_{ij}, & \mbox{if } i\neq j,
\end{array}
\right.$
 $M^{-1}BM$ and $B$ have the same  eigenvalue,
let $x=(x_1, x_2, \ldots, x_n)$ be a positive eigenvector of $M^{-1}BM$
corresponding to $\lambda(B)$.
We assume that one entry, say $x_i$, is equal to 1 and the others are less than or equal to 1, that is, $x_i=1$ and $0<x_k\leq1$ for any $k$. Let $x_j=\max\{x_k|1\leq k\leq n \mbox { and } k\neq i\}$. From
$M^{-1}BMx=\lambda(B)x$,
we have
\vskip.2cm
\hskip0.8cm $\lambda(B)=\lambda(B)x_i=\sum\limits_{k=1}^n\frac{b_{ik}r_kx_k}{r_i}$
$\leq \frac{1}{r_i}(r_i^2+x_j\sum\limits_{k=1}^na_{ik}r_k)=r_i+\frac{x_js_i}{r_i}$, \hskip3.7cm$(1)$
\vskip.2cm
\noindent with equality if and only if $x_k=x_j$ for $1\leq k\leq n \mbox { and } k\neq i$, and

\vskip.2cm
\hskip0.8cm $\lambda(B)x_j=\sum\limits_{k=1}^n\frac{b_{jk}r_kx_k}{r_j}$
 $=r_jx_j+\frac{1}{r_j}\sum\limits_{k=1}^na_{jk}r_kx_k$
$\leq r_jx_j+\frac{1}{r_j}\sum\limits_{k=1}^na_{jk}r_k$
$=r_jx_j+\frac{s_j}{r_j}$.\hskip1.3cm$(2)$
\vskip.2cm
\noindent with equality if and only if $x_k=x_i=1 $ for $1\leq k\leq n \mbox { and } k\neq j$.

\vskip.1cm
From $(1)$ and $(2)$, we get
\vskip.2cm
\hskip0.8cm $\lambda(B)-r_i\leq \frac{x_js_i}{r_i}$, \hskip11.7cm$(3)$

\noindent and

\hskip0.8cm $(\lambda(B)-r_j)x_j\leq \frac{s_j}{r_j}$.  \hskip11.2cm$(4)$
\vskip.1cm
Since $B$ is a positive semidefinite Hermitian matrix, then $\lambda(B)\geq r_1$ by Lemma \ref{lem 2.4}. Thus $\lambda(B)-r_i\geq0$ and $\lambda(B)-r_j\geq0$. From $(3)$ and $(4)$, we get
\vskip.2cm
\hskip0.8cm $\lambda^2(B)-(r_i+r_j)\lambda(B)+r_ir_j-\frac{s_is_j}{r_ir_j}\leq 0$.
\vskip.1cm
Thus we have
\vskip.1cm
\hskip0.8cm$\lambda(B)\leq\frac{r_i+r_j+\sqrt{(r_i-r_j)^2+\frac{4s_is_j}{r_ir_j}}}{2}$.
\vskip.2cm
Moreover, we can see easily the equality holds if and only if row sums of $A$ are all equal.
\end{proof}


In \cite{2013a}, A.D. Maden et al. obtained the result about the sharp bound of the signless Laplacian spectral radius of $G$ as follows Proposition  \ref{error1}.
There exists an error about the condition of the equality holds.

\begin{prop}\label{error1}{\rm(Theorem 6 (b)) in \cite{2013a}}
Let $G$ be a simple and connected graph on $n$ vertices with degrees sequence $\{d_1, d_2, \ldots, d_n\}$ which satisfies $d_1\geq d_2\geq\cdots\geq d_n$ and $q_1(G)$ be the signless Laplacian spectal radius of $G$. Let $s_i=\sum\limits_{v_k\sim v_i}d_k$, where vertices $v_i$ and $v_k$ are adjacent, for $i\in\{1, 2, \ldots, n\}$. Then
\vskip.2cm
\hskip0.8cm $q_1(G)\leq \max\limits_{1\leq i,j\leq n}\left\{\frac{d_i+d_j+\sqrt{(d_i-d_j)^2+\frac{4s_is_j}{d_id_j}}}{2}\right\}$,

\noindent with equality if and only if $G$ is a regular graph and a bipartite semi--regular graph.
\end{prop}

\begin{exam}
The star graph $S_n$ is a bipartite semi--regular with degree sequence $\{n-1, 1, 1, \ldots, 1\}$ and $s_1=n-1$, $s_2=s_3=\cdots=s_n=2n-3$. By Proposition  \ref{error1}, we have  $q_1(S_n)=\max\{\frac{n+\sqrt{n^2+4n-8}}{2}, 2n-2\}=2n-2$. In fact, we know $q_1(S_n)=n$ for directly calculation. It is a contradiction.
\end{exam}
 From the proof of Theorem \ref{thm 3.4}, we can see the equality holds if and only if $G$ is a regular graph. We correct the error Proposition  \ref{error1} in \cite{2013a} as the following corollary.

\begin{cor}\label{cor 2.6}
Let $G$ be a simple and connected graph on $n$ vertices with degrees sequence $\{d_1, d_2, \ldots, d_n\}$ which satisfies $d_1\geq d_2\geq\cdots\geq d_n$ and $q_1(G)$ be the signless Laplacian spectal radius of $G$. Let $s_i=\sum\limits_{v_k\sim v_i}d_k$, where vertices $v_i$ and $v_k$ are adjacent, for $i\in\{1, 2, \ldots, n\}$. Then
\vskip.2cm
\hskip0.8cm $q_1(G)\leq \max\limits_{1\leq i,j\leq n}\left\{\frac{d_i+d_j+\sqrt{(d_i-d_j)^2+\frac{4s_is_j}{d_id_j}}}{2}\right\}$,

\noindent with equality if and only if $G$ is a regular graph.
\end{cor}

\begin{theo}\label{thm 3.5}
Let $A=(a_{ij})$ be an $n\times n$ nonnegative matrix with row sums
$r_1, r_2, \ldots, r_n$, where $r_1\geq r_2\geq\ldots\geq r_n$, $M$ be the largest diagonal element, $N$ be the largest
 non-diagonal element and $\lambda(A)$ be the spectral radius of $A$. Then for $i=1, 2, \ldots, n$,
\vskip.2cm
\hskip0.8cm $\lambda(A)\leq \frac{r_i+M-N+\sqrt{(r_i+N-M)^2+4(i-1)(r_1-r_i)N}}{2}$.  \hskip7cm$(5)$

\vskip.2cm
 Moreover, if $A$ is irreducible,
when $i=1$, then the equality holds if and only if the row sums of $A$ are all equal;
 when $2\leq i\leq n$, then the equality holds if and only if $A$ satisfies the following conditions:
\vskip.1cm
$\rm{(i)}$ $a_{ll}=M$ for $1\leq l\leq i-1$,
\vskip.1cm
$\rm{(ii)}$ $a_{lk}=N$ for $1\leq l, k\leq n$ and $l\neq k$,
\vskip.1cm
$\rm{(iii)}$ $r_1=r_2=\cdots=r_{i-1}\geq r_i=r_{i+1}=\cdots=r_n$.
\end{theo}
\begin{proof}
When $i=1$ or $r_1=r_i$, it is clearly that the inequality is true and the equality holds if and only if the row sums of $A$ are all equal by Lemma \ref{lem 2.2}.
\vskip.1cm
We suppose $r_1>r_i$ and $2\leq i \leq n$.
\vskip.1cm
Let $x=\frac{M-r_i+(2i-3)N+\sqrt{(r_i+N-M)^2+4(i-1)(r_1-r_i)N}}{2(i-1)N}$. It is clear that $x>1$.

Let
$x_j=\left\{
\begin{array}{cc}
  x, & 1\leq j\leq i-1; \\
  1, & i\leq j\leq n.
\end{array}
\right.$
Let $U=diag(x_1, x_2, \ldots, x_n)$ be a diagonal matrix of size $n\times n$. Let $B=U^{-1}AU$ and note that $B$ and $A$ have the same
 eigenvalues, then $\lambda(B)=\lambda(A)$.

 Now we consider the row sums $s_1,s_2, \ldots, s_n$ of matrix $B$.

For $1\leq l\leq i-1$, we have
\vskip.2cm
\hskip0.8cm $s_l=\sum\limits_{j=1}^{n}\frac{x_j}{x_l}a_{lj}$
$=\sum\limits_{j=1}^{i-1}a_{lj}+\frac{1}{x}\sum\limits_{j=i}^{n}a_{lj}$
$=\frac{1}{x}\sum\limits_{j=1}^{n}a_{lj}+(1-\frac{1}{x})\sum\limits_{j=1}^{i-1}a_{lj}$
$=\frac{1}{x}r_l+(1-\frac{1}{x})\sum\limits_{j=1}^{i-1}a_{lj}$,

\noindent and for $i\leq l\leq n$, we have
\vskip.2cm
\hskip0.8cm $s_l=\sum\limits_{j=1}^{n}\frac{x_j}{x_l}a_{lj}$
$=\sum\limits_{j=1}^{n}a_{lj}+(x-1)\sum\limits_{j=1}^{i-1}a_{lj}$
$=r_l+(x-1)\sum\limits_{j=1}^{i-1}a_{lj}$.

Since $A$ is a nonnegative matrix, then $0\leq a_{ij}\leq N$ if $i\neq j$, where $1\leq i,j\leq n$. Thus
\vskip.2cm
\hskip0.8cm $\sum\limits_{j=1}^{i-1}a_{lj}\leq M+(i-2)N$, where $1\leq l\leq i-1$,

\noindent and

\hskip0.8cm $\sum\limits_{j=1}^{i-1}a_{lj}\leq (i-1)N$, where $i\leq l\leq n$.
\vskip.1cm
As $x>1$ and $r_1\geq r_2\geq\cdots\geq r_n$, then for $1\leq l\leq i-1$,
\vskip.2cm
\hskip.8cm $s_l\leq\frac{1}{x}r_l+(1-\frac{1}{x})[M+(i-2)N]$
\vskip.2cm
\hskip1.2cm $=\frac{1}{x}(r_l-M)+M+(1-\frac{1}{x})(i-2)N$
\vskip.2cm
\hskip1.2cm $\leq\frac{1}{x}(r_1-M)+M+(1-\frac{1}{x})(i-2)N$,
\vskip.2cm
\noindent with equality if and only if $(a)$ and $(b)$ hold: $(a)$ $a_{ll}=M$ and $a_{lj}=N$ if $1\leq j \leq i-1$ with $j\neq i$, $(b)$ $r_1=r_l$.
\vskip.2cm
 For $i\leq l\leq n$, we have
\vskip.2cm
\hskip.8cm $s_l\leq r_l+(x-1)(i-1)N$
$\leq r_i+(x-1)(i-1)N$,
\vskip.2cm
\noindent with equality if and only if $(c)$ and $(d)$ hold: $(c)$ $a_{lj}=N$ if $1\leq j \leq i-1$, $(d)$ $r_i=r_l$.

Note that
\vskip.2cm
\hskip0.8cm
$\frac{1}{x}(r_1-M)+M+(1-\frac{1}{x})(i-2)N=r_i+(x-1)(i-1)N=\frac{M+r_i-N+\sqrt{(r_i+N-M)^2+4(i-1)(r_1-r_i)N}}{2}$,
thus
\vskip.2cm
\hskip0.8cm $\lambda(A)=\lambda(B)\leq \max\{s_1, s_2, \ldots, s_n\}\leq\frac{M+r_i-N+\sqrt{(r_i+N-M)^2+4(i-1)(r_1-r_i)N}}{2}$.
\vskip.2cm
When equality holds in $(5)$, then $(a)$ $(b)$ hold for $1\leq l\leq i-1$,
$(c)$ $(d)$ hold for $i\leq l\leq n$. Thus $a_{ll}=M$ for $1\leq l\leq i-1$, $r_1=r_2=\cdots=r_{i-1}>r_i=r_{i+1}=\cdots=r_n$ and all the  non-diagonal elements are equal. Now $\rm{(i)}-\rm{(iii)}$ follow.

Conversely, if $\rm{(i)}-\rm{(iii)}$ hold, it is easy to show that equality holds.
\end{proof}

Similarly, we apply Theorem \ref{thm 3.5} to $Q(G)=D(G)+A(G)$ and note that $r_i=2d_i$ for any $i\in\{1, 2, \ldots, n\}$, $M=d_1$ and $N=1$.
Then we have 

\begin{cor}\label{cor 2.9}{\rm(\cite{2011}, Theorem 3.2)}
Let $G$ be a simple and connected graph with $n$ vertices and degrees sequence $\{d_1, d_2, \ldots, d_n\}$ which satisfies $d_1\geq d_2\geq\cdots\geq d_n$ and $q_1(G)$ be the signless Laplacian spectal radius of $G$. Then for $i=1, 2, \ldots, n$,
\vskip.2cm
\hskip0.8cm $q_1(G)\leq \frac{2d_i+d_1-1+\sqrt{(2d_i+1-d_1)^2+8(i-1)(d_1-d_i)}}{2}$.
\vskip.2cm
Moreover, if $i=1$, then the equality holds if and only if $G$ is a regular graph; if $2\leq i\leq n$, then the equality holds if and only if $G$ is either a regular graph or a bidegreed graph in which $d_1=d_2=\cdots=d_{i-1}=n-1$ and $d_i=d_{i+1}=\cdots=d_n$.
\end{cor}

\section{Bounds on distance signless Laplacian spectral radius}
\hskip.6cm
In this section, we present some sharp bounds on the distance signless Laplacian spectral radius by the above theorems and lemmas in Section 2.

From Lemma \ref{lem 2.2}, we can get the following result.
\begin{theo}\label{thm 2.4}
Let $G$ be a simple connected graph with $n$ vertices and $\mathcal{D}_1$ and $\mathcal{D}_n$ be the largest and least transmissions of $G$, respectively. Then
$\mathcal{D}_n\leq \delta_1(G)\leq \mathcal{D}_1$.
 Moreover, one of the equalities holds if and only if $G$ is a transmission regular graph.
\end{theo}

Note that the $i$th row sum of $\mathcal{D}^Q$ is
$r_i=\mathcal{D}_i+\sum\limits_{j=1}^nd_{ij}=2\mathcal{D}_i$,
we get the following theorem.
\begin{theo}\label{thm 2.3}
Let $G$ be a simple connected graph with $n$ vertices, $\mathcal{D}_1$ and $\mathcal{D}_n$ be the largest and least transmissions of $G$, respectively. Then
$2\mathcal{D}_n\leq \delta_1^Q(G)\leq 2\mathcal{D}_1$.
Moreover, one of the equalities holds if and only if $G$ is a transmission regular graph.
\end{theo}

By Theorem \ref{thm 2.3}, we show the largest distance signless Laplacian spectral radius of $K_n$, $K_{\frac{n}{2}, \frac{n}{2}}$ and $C_n$ as follows.

$\delta_1^Q(K_n)=2(n-1)$, $\delta_1^Q(K_{\frac{n}{2}, \frac{n}{2}})=3n-4$ and
$\delta_1^Q(C_n)=\left\{
\begin{array}{cc}
  \frac{n^2-1}{2}, & \mbox{if } n \mbox{ is } \mbox {odd}, \\[0.2cm]
  \frac{n^2}{2}, & \mbox{if } n \mbox{ is } \mbox {even}.
\end{array}
\right.$

\begin{lem}\label{lem 2.3}{\rm(\cite{2013})}
Let $G$ be a connected graph on $n\geq2$. Then $\delta_1^Q(G)\geq\delta_1^Q(K_n)=2(n-1)$
\vskip.1cm
\noindent and $\delta_i^Q(G)\geq\delta_i^Q(K_n)=n-2$, for all $2\leq i\leq n$. Moreover, $\delta_2^Q(G)\geq\delta_2^Q(K_n)=n-2$ if and only if $G$ is the complete graph $K_n$.
\end{lem}

\begin{theo}\label{thm 2.5}
Let $G$ be a simple connected graph with $n$ vertices, $\delta_1(G)$ be the distance spectral radius of $G$ and $\mathcal{D}_1$ be the largest transmission of $G$. Then $\mathcal{D}_1\leq\delta_1^Q(G)\leq \delta_1(G)+\mathcal{D}_1$. Moreover, $\delta_1^Q(G)=\delta_1(G)+\mathcal{D}_1$ if and only if $G$ is a transmission regular graph.
\end{theo}
\begin{proof}
Denote the spectrum of $\mathcal{D}^Q$ by $sp(\mathcal{D}^Q)=(\delta_1^Q(G), \delta_2^Q(G), \ldots, \delta_2^Q(G))$ and the
transmissions $\mathcal{D}_1, \mathcal{D}_2, \ldots, \mathcal{D}_n$ of $G$ satisfy $\mathcal{D}_1\geq \mathcal{D}_2 \geq\cdots\geq \mathcal{D}_n$. Note that $\delta_n^Q(G)\geq \delta_n^Q(K_n)=n-2\geq0$ for $n\geq 2$ by Lemma \ref{lem 2.3}, then $\mathcal{D}^Q$ is a positive semidefinite Hermitian matrix. By Lemma \ref{lem 2.4}, we have $\delta_1^Q(G)\geq \mathcal{D}_1$.

By Theorem \ref{thm 3.1}, we directly get the $\delta_1^Q(G)\leq \delta_1(G)+\mathcal{D}_1$ with equality if and only if $G$ is a transmission regular graph.
\end{proof}

\begin{defn}\label{defn 2.1}{\rm(\cite{2010b})}
Let $G$ be a simple connected graph with $n$ vertices, a distance matrix $\mathcal{D}=(d_{ij})$ and the transmission sequence $\{ \mathcal{D}_1, \mathcal{D}_2, \ldots, \mathcal{D}_n\}$ with $\mathcal{D}_1\geq \mathcal{D}_2 \geq\cdots\geq \mathcal{D}_n$. Then the second distance degree of $v_i$, denoted by $T_i$, is given by $T_i=\sum\limits_{i=1}^{n}d_{ij}\mathcal{D}_i$.
\end{defn}

We can easily obtain Theorems \ref{thm 3.7} and \ref{thm 3.8} from Theorems \ref{thm 3.4} and \ref{thm 3.5}, respectively.
\begin{theo}\label{thm 3.7}
Let $G$ be a simple connected graph with $n$ vertices, $\{\mathcal{D}_1, \mathcal{D}_2, \cdots, \mathcal{D}_n\}$ be the transmission sequence of $G$ with
$\mathcal{D}_1\geq \mathcal{D}_2 \geq\cdots \geq \mathcal{D}_n$ and $T_1, T_2,\ldots, T_n$ be the second distance degree sequence of $G$. Then
\vskip.2cm
\hskip0.8cm $\delta_1^Q(G)\leq \max\limits_{1\leq i, j\leq n}\left\{\frac{\mathcal{\mathcal{D}}_i+\mathcal{D}_j+\sqrt{(\mathcal{D}_i-\mathcal{D}_j)^2+\frac{4T_iT_j}{\mathcal{D}_i\mathcal{D}_j}}}{2}\right\}$,
\vskip.2cm
\noindent with equality if and only if $G$ is a transmission regular graph.
\end{theo}
\begin{proof}
Note that $\mathcal{D}^{Q}=Tr(G)+\mathcal{D}(G)$, apply Theorem \ref{thm 3.4} to $\mathcal{D}^{Q}$. Since $r_i=\mathcal{D}_i$ and $S_i=T_i$ for $i=1, 2, \ldots, n$, we have the desired upper bound for $\delta_1^Q(G)$.
\end{proof}

\begin{theo}\label{thm 3.8}
Let $G$ be a simple connected graph with $n$ vertices, $\{\mathcal{D}_1, \mathcal{D}_2, \ldots, \mathcal{D}_n\}$ be the transmission sequence of $G$ with $\mathcal{D}_1\geq\mathcal{D}_2\geq\cdots\geq\mathcal{D}_n$ and $d$ be the diameter of $G$. Then for $i=1, 2, \ldots, n$,
\vskip.2cm
\hskip0.8cm $\delta_1^Q(G)\leq \frac{2\mathcal{D}_i+\mathcal{D}_1-d+\sqrt{(2\mathcal{D}_i+d-\mathcal{D}_1)^2+8(i-1)(\mathcal{D}_1-\mathcal{D}_i)d}}{2}$.
\vskip.2cm
\noindent Moreover, if $i=1$, the equality holds if and only if $G$ is a transmission regular graph;
and if $2\leq i\leq n$, the equality holds if and only if $G\cong K_n$.
\end{theo}
\begin{proof}
Apply Theorem \ref{thm 3.5} to $\mathcal{D}^Q(G)$. Since $M=\mathcal{D}_1$, $N=d$ and $r_i=2\mathcal{D}_i$ for $i=1, 2, \ldots, n$,
 we have the desired upper bound for $\delta_1^Q(G)$. If $i=1$, the equality holds if and only if $G$ is a transmission regular graph;
and if $2\leq i\leq n$, the equality holds if and only if $G\cong K_n$ by the conditions $\rm{(i)}-\rm{(iii)}$ in Theorem \ref{thm 3.5}.
\end{proof}

\begin{theo}\label{thm 3.2}
Let $G$ be a simple connected graph with $n$ vertices, $\{\mathcal{D}_1, \mathcal{D}_2, \ldots, \mathcal{D}_n\}$ be the transmission sequence of $G$ with $\mathcal{D}_1\geq \mathcal{D}_2 \geq\cdots \geq \mathcal{D}_n$ and $\{T_1, T_2, \ldots, T_n\}$ be the second distance degree sequence of $G$. Then
\vskip.1cm
\hskip0.8cm$\min\left\{\mathcal{D}_i+\frac{T_i}{\mathcal{D}_{i}}: 1\leq i\leq n\right\}\leq \delta_1^Q(G)\leq \max\left\{\mathcal{D}_i+\frac{T_i}{\mathcal{D}_{i}}: 1\leq i\leq n\right\}$.
\vskip.1cm
\noindent Moreover, any equality holds if and only if $G$ has the same value $\mathcal{D}_i+\frac{T_i}{\mathcal{D}_{i}}$ for all $i$.
\end{theo}
\begin{proof}
Note that $\mathcal{D}^Q=(q_{ij})$, while
$q_{ij}=\left\{
\begin{array}{cc}
 \mathcal{D}_i , & \mbox{if } i=j; \\[0.2mm]
 d_{ij}, & \mbox{if } i\neq j.
\end{array}
\right.$

 Then the $(i, j)$--entry of
 $Tr(G)^{-1}\mathcal{D}^QTr(G)$ is $\frac{q_{ij}\mathcal{D}_j}{\mathcal{D}_i}$ and the row sum of $Tr(G)^{-1}\mathcal{D}^QTr(G)$ is
 \vskip.1cm
\noindent $r_i(Tr(G)^{-1}\mathcal{D}^QTr(G))=\mathcal{D}_i+\frac{T_i}{\mathcal{D}_i}$ for $i=1, 2, \ldots, n$.
\vskip.2cm
It is easy to show that $Tr(G)^{-1}\mathcal{D}^QTr(G)$ is a nonnegative and irreducible $n\times n$ matrix with
 spectral radius $\delta_1^Q(G)$. By Lemma \ref{lem 2.2}, the desired results hold.
\end{proof}

\begin{theo}\label{thm 3.3}
Let $G$ be a simple connected graph with $n$ vertices, $\{\mathcal{D}_1, \mathcal{D}_2,\ldots, \mathcal{D}_n\}$ be the transmission sequence of $G$ with $\mathcal{D}_1\geq \mathcal{D}_2 \geq\cdots \geq \mathcal{D}_n$ and $\{T_1, T_2,\ldots, T_n\}$ be the second distance degree sequence of $G$. Then
\vskip.2cm
\hskip0.8cm $\min\left\{\sqrt{2T_i+2\mathcal{D}_i^2}: 1\leq i\leq n\right\}\leq \delta_1^Q(G)\leq \max\left\{\sqrt{2T_i+2\mathcal{D}_i^2}: 1\leq i\leq n\right\}$.
\vskip.2cm
\noindent Moreover, if $(\mathcal{D}^Q)^2$ is irreducible, then each equality holds if and only if $G$ has the same value
\vskip.2cm
\noindent $T_i+\mathcal{D}_i^2$ for all $i$; and if $(\mathcal{D}^Q)^2$ is reducible, then any equality holds if and only if there exist the permutation matrix $P$ such that
\vskip.2cm
\hskip0.8cm $P^T\mathcal{D}^QP=\begin{bmatrix}O_r & \mathcal{D}^Q_1\\\mathcal{D}^Q_2 & O_{n-r}\end{bmatrix}$,
\vskip.2cm
\noindent and $\mathcal{D}_{\sigma(1)}=\mathcal{D}_{\sigma(2)}=\cdots=\mathcal{D}_{\sigma(r)}$, $\mathcal{D}_{\sigma(r+1)}=\mathcal{D}_{\sigma(r+2)}=\cdots=\mathcal{D}_{\sigma(n)}$, where $\sigma$ is a permutation
\vskip.1cm
\noindent on the set $\{1,2, \ldots, n\}$ which corresponds to the permutation matrix $P$.
\end{theo}
\begin{proof}
Since $\mathcal{D}^Q=(q_{ij})$, then $((\mathcal{D}^Q)^2)_{ij}=\sum\limits_{k=1}^nq_{ik}q_{kj}$ and the row sum of $(\mathcal{D}^Q)^2$ is
\vskip.2cm
\hskip0.8cm $r_i((\mathcal{D}^Q)^2)=\sum\limits_{j=1}^n\sum\limits_{k=1}^nq_{ik}q_{kj}$
$=\sum\limits_{k=1}^nq_{ik}\sum\limits_{j=1}^nq_{kj}$
$=2\sum\limits_{k=1}^nq_{ik}\mathcal{D}_k$
$=2T_i+2\mathcal{D}_i^2$.

 Let $x$ be the unit Perron vector corresponding to $\delta_1^Q(G)$. Then $\mathcal{D}^Qx=\delta_1^Q(G)x$ and $(\mathcal{D}^Q)^2x=(\delta_1^Q(G))^2x$. By Lemma \ref{lem 2.2},
\vskip.2cm
\hskip0.8cm $\min\left\{2T_i+2\mathcal{D}_i^2: 1\leq i\leq n\right\}\leq (\delta_1^Q(G))^2=\delta_1((\mathcal{D}^Q)^2)\leq \max\left\{2T_i+2\mathcal{D}_i^2: 1\leq i\leq n\right\}$.
\vskip.2cm
 Thus $\min\left\{\sqrt{2T_i+2\mathcal{D}_i^2}: 1\leq i\leq n\right\}\leq \delta_1^Q(G)\leq \max\left\{\sqrt{2T_i+2\mathcal{D}_i^2}: 1\leq i\leq n\right\}$.
\vskip.2cm
If $(\mathcal{D}^Q)^2$ is irreducible, then we assume that $G$ has the same value $T_i+\mathcal{D}_i^2$ for all $i$. Thus
$\min\left\{\sqrt{2T_i+2\mathcal{D}_i^2}: 1\leq i\leq n\right\}=\max\left\{\sqrt{2T_i+2\mathcal{D}_i^2}: 1\leq i\leq n\right\}$.
\vskip.1cm
Both of the equalities hold.
\vskip.1cm
Conversely, if one of the equalities holds, that is,
\vskip.2cm
\hskip0.8cm $(\delta_1^Q(G))^2=\min\left\{2T_i+2\mathcal{D}_i^2 : 1\leq i\leq n\right\}$ or $(\delta_1^Q(G))^2=\min\left\{2T_i+2\mathcal{D}_i^2 : 1\leq i\leq n\right\}$,
\vskip.2cm
\noindent by Lemma \ref{lem 2.2}, the row sums of $\mathcal{D}^Q$, $r_i((\mathcal{D}^Q)^2)=2T_i+2\mathcal{D}_i^2$, where $1\leq i\leq n$, are all equal.
\vskip.1cm
\noindent Thus $G$ has the same value $T_i+\mathcal{D}_i^2$ for all $i$.
\vskip.2cm
If $(\mathcal{D}^Q)^2$ is reducible, then any equality holds if and only if there exist the permutation matrix $P$ such that
\vskip.2cm
\hskip0.8cm $P^T\mathcal{D}^QP=\begin{bmatrix}O_r & \mathcal{D}^Q_1\\\mathcal{D}^Q_2 & O_{n-r}\end{bmatrix}$,
\vskip.2cm
\noindent and $\mathcal{D}_{\sigma(1)}=\mathcal{D}_{\sigma(2)}=\cdots=\mathcal{D}_{\sigma(r)}$, $\mathcal{D}_{\sigma(r+1)}=\mathcal{D}_{\sigma(r+2)}=\cdots=\mathcal{D}_{\sigma(n)}$, where $\sigma$ is a permutation
\vskip.1cm
\noindent on the set $\{1,2, \ldots, n\}$ which corresponds to the permutation matrix $P$.
\end{proof}


\end{document}